\documentclass[12pt]{article}

\pdfoutput=1

\usepackage{amsmath, amssymb, amsthm, enumerate, fullpage, bm}
\usepackage{verbatim, enumitem, bbm, etoolbox}
\usepackage{latexsym}

\apptocmd{\lim}{\limits}{}{}

\theoremstyle{definition}
\newtheorem{thm}{Theorem}[section]
\newtheorem{theorem}[thm]{Theorem}

\newtheorem{lemma}[thm]{Lemma}

\numberwithin{subcase}{case}

\theoremstyle{definition}
\newtheorem{definition}[thm]{Definition}

\def\forkindep{\mathrel{\raise0.2ex\hbox{\ooalign{\hidewidth$\vert$\hidewidth\cr\raise-0.9ex\hbox{$\smile$}}}}}

\def\Ind{\setbox0=\hbox{$x$}\kern\wd0\hbox to 0pt{\hss$\mid$\hss}
	\lower.9\ht0\hbox to 0pt{\hss$\smile$\hss}\kern\wd0}

\def\Notind{\setbox0=\hbox{$x$}\kern\wd0\hbox to 0pt{\mathchardef
		\nn=12854\hss$\nn$\kern1.4\wd0\hss}\hbox to 0pt{\hss$\mid$\hss}\lower.9\ht0
	\hbox to 0pt{\hss$\smile$\hss}\kern\wd0}

\newcommand{\HC}{\textrm{HC}}

\def\phi{\varphi}

\def\lg{{\rm lg}}
\def\<{\langle}
\def\>{\rangle}

\makeatletter
\def\blfootnote{\xdef\@thefnmark{}\@footnotetext}
\makeatother

\begin{document}	

	\bibliographystyle{plain}
	
	\author{Douglas Ulrich\!\!\
	\thanks{Partially supported
by Laskowski's NSF grant DMS-1308546.}\\
Department of Mathematics\\University of California, Irvine}
	\title{A streamlined proof of $\mathfrak{p}=\mathfrak{t}$}
	\date{\today} 
	
	\blfootnote{\emph{Key Words and Phrase:} Uncountable model theory, cardinal characteristics of the continuum.}
	
	\maketitle
	
	
\begin{abstract}
We streamline Malliaris and Shelah's proof \cite{pEqualsTref} that $\mathfrak{p} = \mathfrak{t}$. In particular, we replace cofinality spectrum problems with models of $ZFC^-$, and we eliminate the use of peculiar cuts.
\end{abstract}

\section{Introduction}
In \cite{pEqualsTref}, Malliaris and Shelah introduce the notion of cofinality spectrum problems; these are essentially models of a weak fragment of arithmetic. To each cofinality spectrum problem $\mathbf{s}$ they associate cardinals $\mathfrak{p}_{\mathbf{s}}$ and $\mathfrak{t}_{\mathbf{s}}$, which measure certain saturation properties of $\mathbf{s}$. In their Central Theorem 9.1, they prove that $\mathfrak{t}_{\mathbf{s}} \leq \mathfrak{p}_{\mathbf{s}}$; moreover, under mild conditions on $\mathbf{s}$ (in \cite{pEqualsT2}, they note that exponentiation is sufficient), equality occurs. Malliaris and Shelah then derive two applications of this: first, they prove that $SOP_2$ theories are maximal in Keisler's order, and second, they prove that $\mathfrak{p} = \mathfrak{t}$. The latter application resolves the most longstanding open problem in the theory of cardinal invariants of the continuum, and we give a self-contained treatment in this paper. We discuss the first application in \cite{InterpOrdersUlrich}.

The main difficulty encountered by readers of \cite{pEqualsTref} is in the definition of cofinality spectrum problems; these are rather convoluted objects, but in fact they are not necessary to the proof. All that is needed is some fragment of $ZFC$ with transitive set models. $ZFC^-$ $(ZFC$ without powerset) is convenient for our purposes. The reader comfortable with mild large cardinals should feel free to replace $ZFC^-$ by $ZFC$ (or more). 

A model of $ZFC^-$ is $\omega$-nonstandard if it contains nonstandard natural numbers. To every $\omega$-nonstandard $\hat{V} \models ZFC^-$ we will associate a pair of cardinal invariants $\mathfrak{p}_{\hat{V}}$ and $\mathfrak{t}_{\hat{V}}$. The reader familiar with cofinality spectrum problems may verify that any $\omega$-nonstandard $\hat{V} \models ZFC^-$ determines a cofinality spectrum problem $\mathbf{s}$, and that $\mathfrak{p}_{\hat{V}} = \mathfrak{p}_{\mathbf{s}}$ and $\mathfrak{t}_{\hat{V}} = \mathfrak{t}_{\mathbf{s}}$, following the proof of Claim 10.19 of \cite{pEqualsTref}.

We now give an overview of our proof that $\mathfrak{p}= \mathfrak{t}$. First, in Section~\ref{pVLeqTvSec}, we show that for every $\omega$-nonstandard $\hat{V} \models ZFC^-$, $\mathfrak{p}_{\hat{V}} \leq \mathfrak{t}_{\hat{V}}$; Malliaris and Shelah prove this in \cite{pEqualsTref} in the context of ultrapower embeddings, and in \cite{pEqualsT2} they note that it holds for cofinality spectrum problems with exponentiation. We also give a useful condition for when a partial type $p(x)$ over $\hat{V}$ of cardinality less than $\mathfrak{p}_{\hat{V}}$ is realized in $\hat{V}$. Next, in Section~\ref{pVEqualsTvSec} we show that for every $\omega$-nonstandard $\hat{V} \models ZFC^-$, $\mathfrak{p}_{\hat{V}} = \mathfrak{t}_{\hat{V}}$. 

Finally, in Section~\ref{CardInvariants}, we prove $\mathfrak{p}=\mathfrak{t}$, loosely following Malliaris and Shelah: first, note that it follows immediately from the definitions that $\mathfrak{p} \leq \mathfrak{t}$, so we suppose that $\mathfrak{p} < \mathfrak{t}$ to get a contradiction. We are free to suppose that $\mathfrak{t} = 2^{\aleph_0} = 2^{<\mathfrak{t}}$, since we can Levy-collapse $2^{<\mathfrak{t}}$ to $\mathfrak{t}$ without adding sequences of length less than $\mathfrak{t}$. We are then able to construct a sufficiently generic ultrafilter $\mathcal{U}$ on $\mathcal{P}(\omega)$, such that if we set $\hat{V} = V^\omega/\mathcal{U}$ for some or any transitive $V \models ZFC^-$, then $\mathfrak{p}_{\hat{V}} \leq \mathfrak{p}$ and $\mathfrak{t} \leq \mathfrak{t}_{\hat{V}}$. This contradicts our earlier result that $\mathfrak{p}_{\hat{V}} =\mathfrak{t}_{\hat{V}}$. We manage to avoid reference to a hard theorem of Shelah involving peculiar cuts \cite{ptComment}.

We remark that Moranarocca gives a proof of $\mathfrak{p} = \mathfrak{t}$ in \cite{pEqualsTproof2}, following an unpublished proof sketch of J. Steprans; also, Fremlins has posted a proof on his website \cite{Fremlin}, also based on Stepran's sketch. The main difference from Malliaris and Shelah's proof is that Stepran replaces cofinality spectrum problems by ultrapower embeddings. We prefer working with models of set theory, since the ultrapower machinery introduces unneeded notational overhead. Both of these proofs \cite{pEqualsTproof2} \cite{Fremlin} use peculiar cuts.

\section{$\mathfrak{p}_{\hat{V}} \leq \mathfrak{t}_{\hat{V}}$}\label{pVLeqTvSec}
We begin with some formalities. $ZFC^-$ is $ZFC$ without powerset,  but with replacement strengthened to collection, and with choice strengthened to the well-ordering principle; we consider this the standard definition, following \cite{ZFCminus}.

As some notational conventions, $\hat{V}$ will denote a model of $ZFC$. Whenever $\hat{V} \models ZFC^-$, we will identify $HF$ (the hereditarily finite sets) with its copy in $\hat{V}$; for example, we identify each natural number $n < \omega$ with its copy in $\hat{V}$. Other elements of $\hat{V}$ will usually be decorated with a hat, for instance we write $\hat{\omega}$ rather than $(\omega)^{\hat{V}}$; but sometimes readability takes precedence. Given $X \subseteq \hat{V}$, we say that $X$ is an internal subset of $\hat{V}$ if there is some $\hat{X} \in \hat{V}$ such that $X = \{\hat{y} \in \hat{V}: \hat{y} \hat{\in} \hat{X}\}$. In this case, we identify $X$ with $\hat{X}$ and will write that $X \in \hat{V}$. 

We say that $\hat{V}$ is $\omega$-standard, or is an $\omega$-model, if $\hat{\omega} = \omega$ (i.e. every natural number of $\hat{V}$ has finitely many predecessors). $\hat{V}$ will only ever denote non $\omega$-models. We say that $X \subseteq \hat{V}$ is pseudofinite if there is some $\hat{X} \in \hat{V}$, finite in the sense of $\hat{V}$, such that $X \subseteq \hat{X}$. Thus if $\hat{X} \in \hat{V}$, then $\hat{X}$ is pseudofinite if and only if it is finite in the sense of $\hat{V}$.

We now make the key definitions.
\begin{definition}
	If $(L, <)$ is a linear order, and $\kappa, \theta$ are infinite regular cardinals, then a $(\kappa, \theta)$-pre-cut in $L$ is a pair of sequences $(\overline{a}, \overline{b}) = (a_\alpha: \alpha < \kappa)$, $(b_\beta: \beta < \theta)$ from $L$, such that for all $\alpha < \alpha'$, $\beta < \beta'$, we have $a_\alpha < a_{\alpha'} < b_{\beta'} < b_{\beta}$. $(\overline{a}, \overline{b})$ is a cut if there is no $c \in L$ with $a_\alpha < c < b_\beta$ for all $\alpha, \beta$. Let the cut spectrum of $(L, <)$ be $\mathcal{C}(L, <) := \{(\kappa, \theta): L \mbox{ admits a } (\kappa, \theta) \mbox{ cut}\}$. Define $\mbox{cut}(L, <) = \mbox{min}\{\kappa + \theta:  (\kappa, \theta) \in \mathcal{C}(L, <)\}$.
	
	By a tree $T$ we mean a partially ordered set $(T, <)$ with meets and a minimum element $0_T$, such that the predecessors of every element are linearly-ordered. Given a tree $(T, <)$ define $\mbox{tree-tops}(T)$ to be the least (necessarily regular) $\kappa$ such that there is an increasing sequence $(s_\alpha: \alpha < \kappa)$ from $T$ with no upper bound in $T$.
	
	Suppose $\hat{V}$ is an $\omega$-nonstandard model of $ZFC^-$. Then define $\mathcal{C}_{\hat{V}}= \mathcal{C}(\hat{\omega}, \hat{<})$, and define  $\mathfrak{p}_{\hat{V}} = \mbox{cut}(\hat{\omega}, \hat{<})$. Also, let $\mathfrak{t}_{\hat{V}}$ be the minimum over all $\hat{n} < \hat{\omega}$ of $\mbox{tree-tops}(\hat{n}^{<\hat{n}}, \hat{\subset})$. 
\end{definition}

Unraveling the definitions, $\mathbf{t}_{\hat{V}}$ is the least $\kappa$ such that there is some $\hat{n}< \hat{\omega}$ and some increasing sequence $(\hat{s}_\alpha: \alpha < \kappa)$ from $\hat{n}^{<\hat{n}}$, with no upper bound in $\hat{n}^{<\hat{n}}$. Equivalently, $\mathbf{t}_{\hat{V}}$ is the least $\kappa$ such that there is some $\hat{n}< \hat{\omega}$ and some increasing sequence $(\hat{s}_\alpha: \alpha < \kappa)$ from $\hat{n}^{<\hat{n}}$, with no upper bound in $\hat{\omega}^{<\hat{\omega}}$; this is because if $\hat{s}$ is any upper bound, then $\hat{s} \restriction_{\hat{m}}$ is an upper bound in $\hat{n}^{<\hat{n}}$, where $\hat{m}$ is the largest number below $\hat{n}$ so that $\hat{s} \restriction_{\hat{m}} \in \hat{n}^{\hat{m}}$.

The following lemma is a component of Shelah's proof in \cite{ShelahIso} that SOP theories are maximal in Keisler's order. It need not hold for cofinality spectrum problems. In Section 10 of \cite{pEqualsTref}, Malliaris and Shelah derive the lemma in the context of ultrapower embeddings, following \cite{ShelahIso}. In \cite{pEqualsT2}, Malliaris and Shelah comment that cofinality spectrum problems with exponentiation are enough.

\begin{lemma}\label{pLeqT}
	Suppose $\hat{V} \models ZFC^-$ is $\omega$-nonstandard. Then 
	$\mathfrak{p}_{\hat{V}} \leq \mathfrak{t}_{\hat{V}}$. In fact, $(\mathfrak{t}_{\hat{V}}, \mathfrak{t}_{\hat{V}}) \in \mathcal{C}_{\hat{V}}$.
\end{lemma}
\begin{proof}
	Suppose $(s_\alpha: \alpha < \kappa)$ is an increasing sequence from $\hat{n_*}^{<\hat{n_*}}$ with no upper bound, where $\kappa$ is regular. We show $(\kappa, \kappa) \in \mathcal{C}_{\hat{V}}$.
	
	Let $\hat{<}_{lex}$ be the lexicographic ordering on $\hat{n}_*^{<\hat{n}_*}$.
	
	Note that if $s \in \hat{T}$, then $s_\alpha \,^\frown(0) \, \hat{\leq}_{lex} \, s \hat{\leq}_{lex} s_\alpha \,^\frown(\hat{n}_*-1)$ if and only if $s_\alpha \subseteq s$. Since $(s_\alpha: \alpha < \kappa)$ is unbounded, it follows that $(s_\alpha\,^\frown(0): \alpha < \kappa)$ and $(s_\alpha\,^\frown(\hat{n}_*-1): \alpha < \kappa)$ form a $(\kappa, \kappa)$-cut in $(\hat{n}_*^{<\hat{n}_*}, (\hat{<}_{lex})_*)$.
	
	In $\hat{V}$, let $\hat{\sigma}: (\hat{n}_*^{<\hat{n}_*}, \hat{<}_{lex}) \to (|\hat{n}_*^{<\hat{n}_*}|, \hat{<})$ be the order preserving bijection. Then $(\hat{\sigma}(s_\alpha\,^\frown(0)): \alpha < \kappa)$ and $(\hat{\sigma}(s_\alpha\,^\frown(\hat{n}_*-1)): \alpha < \kappa)$ witness that $(\kappa, \kappa) \in \mathcal{C}(\hat{\omega}, \hat{V})$.
\end{proof}

The following corresponds to Claim 2.14 of \cite{pEqualsTref}.

\begin{lemma}\label{definableTreeTops}
	Suppose $\hat{V} \models ZFC^-$ is $\omega$-nonstandard.
	Suppose $(\hat{T}, \hat{<})$ is a pseudofinite tree in $\hat{V}$. Then $\mbox{tree-tops}(\hat{T}, \hat{<}) \geq \mathfrak{t}_{\hat{V}}$.
\end{lemma}
\begin{proof}
	There is in $\hat{V}$ a subtree of $\hat{\omega}^{<\hat{\omega}}$ which is isomorphic to $\hat{T}$; so we can suppose that $\hat{T}$ is a subtree of $\hat{\omega}^{<\hat{\omega}}$. Then $\hat{T}$ is a subtree of $\hat{n}_*^{<\hat{n}_*}$ for some $\hat{n}_* < \hat{\omega}$.
	
	Now suppose $(s_\alpha: \alpha < \kappa)$ is an increasing sequence from $\hat{T}$  with $\kappa < \mathfrak{t}_{\hat{V}}$; we show there is an upper bound in $\hat{T}$. To see this let $s_+$ be an upper bound of $(s_\alpha: \alpha < \kappa)$ in $\hat{\omega}^{<\hat{\omega}}$, and let $\hat{n}$ be largest so that $s_+ \restriction_{\hat{n}} \, \hat{\in} \, \hat{T}$; and let $s = s_+ \restriction_{\hat{n}}$. 
\end{proof}

The following theorem corresponds to Theorem 4.1 of \cite{pEqualsTref}, although there the authors must also assume $\lambda < \mathfrak{t}_{\hat{V}}$ in the absence of Lemma~\ref{pLeqT}. Note that since models of $ZFC^-$ admit pairing functions, there is no loss in only considering types in a single variable, in which each formula has only a singleton parameter.

\begin{theorem}\label{localSaturation}
	Suppose $\hat{V} \models ZFC^-$ is $\omega$-nonstandard.
	Suppose  $p(x)= (\phi_\alpha(x, a_\alpha): \alpha < \lambda)$ is a partial type over $\hat{V}$ of cardinality $\lambda < \mathfrak{p}_{\hat{V}}$. Suppose $\hat{X} \in \hat{V}$ is pseudofinite, and $\phi_0(x)$ is $``x \in \hat{X}$." Then $p(x)$ is realized in $\hat{V}$.
\end{theorem}

\begin{proof}
	Obviously this is true when $\lambda$ is finite.
	
	Suppose the lemma is true for all $\lambda' < \lambda$; we show it is true for $\lambda$. This suffices. Write $\hat{n}_* = |\hat{X}|$.
	
	We choose $(s_\alpha: \alpha \leq \lambda)$ an increasing sequence from $\hat{X}^{<\hat{n}_*}$, such that if we let $\hat{n}_\alpha = \hat{\lg}(s_\alpha)$, then for all $\beta < \alpha < \lambda$ and for all $\hat{n}_\beta \leq \hat{n} < \hat{n}_\alpha$, $\hat{V} \models \phi_\beta(s_\alpha(\hat{n}), a_\beta)$. Obviously then $s_\lambda(\hat{n}_\lambda - 1)$ will realize $p(x)$.
	
	Let $s_0 = \emptyset$. At successor stage $\alpha$, just use the hypothesis for $\lambda' = |\alpha| < \lambda$.
	
	Suppose we have defined $(s_\alpha: \alpha < \delta)$ where $\delta \leq \lambda$. Using $|\delta| < \mathfrak{p}_{\hat{V}} \leq \mathfrak{t}_{\hat{V}}$ (by Lemma~\ref{pLeqT}), we may apply Lemma~\ref{definableTreeTops} to choose $s_+ \in \hat{X}^{<\hat{n}_*}$, an upper bound of $(s_\alpha: \alpha < \delta)$.
	
	Let $\hat{m}_0 = \hat{\lg}(s_+)$. For $\beta \leq \delta$ we will define $\hat{m}_\beta$ so that for all $\alpha < \delta$, and for all $\beta < \beta' < \delta$, $\hat{n}_\alpha < \hat{m}_{\beta'} < \hat{m}_\beta$, and further for every $\beta \leq \delta$, we have that for every $\beta' < \beta$ and for every $\hat{n}_{\beta'} \leq \hat{n} < \hat{m}_\beta$, $\hat{V} \models \phi_{\beta'}(s_+(\hat{n}), a_{\beta'})$. Note once we finish we can set $s_\delta = s_+ \restriction_{\hat{m}_\delta}$.
	
	Having defined $\hat{m}_\beta$ for $\beta < \delta$, let $\hat{m}_{\beta+1}$ be the greatest $\hat{m} < \hat{m}_\beta$ such that for all $\hat{n} < \hat{m}$, $\hat{V} \models \phi_\beta(s_+(\hat{n}), a_\beta)$; this works. Having defined $\hat{m}_\beta$ for all $\beta < \delta' \leq \delta$, since $\delta' \leq \delta \leq \lambda < \mathfrak{p}_{\hat{V}}$ we can choose $\hat{m}_{\delta'}$ with $\hat{n}_\alpha < \hat{m}_{\delta'} < \hat{m}_{\beta}$ for all $\alpha < \delta$, $\beta < \delta'$.
	
	This concludes the construction.
\end{proof}

\section{$\mathfrak{p}_{\hat{V}} = \mathfrak{t}_{\hat{V}}$}\label{pVEqualsTvSec}

In this section, we prove the following theorem. It corresponds to Central Theorem 9.1 of \cite{pEqualsTref}.

\begin{theorem}\label{pVEqualsTv}Suppose $\hat{V} \models ZFC^-$ is $\omega$-nonstandard. Then $\mathfrak{p}_{\hat{V}} = \mathfrak{t}_{\hat{V}}$.
\end{theorem}

Fix $\hat{V}$ for the rest of the section.

We begin with the following theorem; it corresponds to Theorem 3.1 of \cite{pEqualsTref}.

\begin{theorem}\label{lcfWellDefined} Suppose $\kappa < \mbox{min}(\mathfrak{p}_{\hat{V}}^+, \mathfrak{t}_{\hat{V}})$ is regular. Then there is a unique regular cardinal $\lambda$ with $(\kappa, \lambda) \in \mathcal{C}_{\hat{V}}$; moreover this $\lambda$ is also unique with the property that $(\lambda, \kappa) \in \mathcal{C}_{\hat{V}}$.
\end{theorem}
\begin{proof}
	We first show that there exist $\lambda_0, \lambda_1$ regular cardinals with $(\kappa, \lambda_0) \in \mathcal{C}_{\hat{V}}$ and $(\lambda_1, \kappa) \in \mathcal{C}_{\hat{V}}$. We will then show that $\lambda_0 = \lambda_1$, which suffices to prove the theorem.
	
	For $\lambda_0$: pick $\hat{n}_*$ nonstandard, and note that by Lemma~\ref{definableTreeTops} applied to the tree $(\hat{n}_*, \hat{<})$ we can choose $(\hat{n}_\alpha: \alpha < \kappa)$ a strictly increasing sequence below $\hat{n}_*$. Let $(\hat{m}_\beta: \beta < \beta_*)$ be any strictly decreasing sequence in $(\hat{n_*}, \hat{<})$, cofinal above $(\hat{n}_\alpha: \alpha < \kappa)$, and then discard elements to replace $\beta_*$ by $\mbox{cof}(\beta_*) =: \lambda_0$. 
	
	For $\lambda_1$: I claim that we can define $(\hat{n}'_\alpha: \alpha < \kappa)$, a strictly decreasing sequence of nonstandard numbers from $\hat{\omega}$. To see that we can do this: first let $\hat{n}'_0$ be an arbitrary nonstandard natural number. Having defined $\hat{n}'_{\alpha}$, let $\hat{n}'_{\alpha+1} = \hat{n}'_\alpha - 1$. Having defined $\hat{n}'_\alpha$ for all $\alpha < \delta$ where $\delta < \kappa$ is a limit, consider the pre-cut $(n: n < \omega), (\hat{n}'_\alpha: \alpha < \delta)$. Since $\omega + \delta < \kappa \leq \mathfrak{p}_{\hat{V}}$ this cannot be a cut, so choose $\hat{n}'_\delta$ in the gap. Having constructed $\hat{n}'_\alpha$ for each $\alpha < \kappa$, we can as in the previous paragraph choose a regular $\lambda_1$ and a strictly increasing sequence $(\hat{m}'_\gamma: \gamma < \lambda_1)$, cofinal below $(\hat{n}'_\alpha: \alpha < \kappa)$.
	
	Now to show $\lambda_0 = \lambda_1$: first, by possibly increasing $\hat{n}_*$, we can suppose $\hat{n}_* > \hat{n}'_0$, and thus each $\hat{n}_\alpha, \hat{n}'_\alpha, \hat{m}_\beta, \hat{m}'_\beta < \hat{n}_*$. Let $(\hat{T}, \hat{<})$ be the tree of all sequences $s \in (\hat{n}_* \times \hat{n}_*)^{<\hat{n}_*}$, such that that for all $\hat{n} < \hat{m} < \hat{\lg}(s)$, $s(\hat{n})(0) < s(\hat{m})(0) < s(\hat{m})(1) < s(\hat{n})(0)$.
	
	We now choose a strictly increasing sequence $(s_\alpha: \alpha < \kappa)$ from $\hat{T}$ such that for each $\alpha < \kappa$, if we set $\hat{a}_\alpha = \hat{\lg}(s_\alpha)$, then $s_{\alpha}(\hat{a}_\alpha-1) = (\hat{n}_\alpha, \hat{n}'_\alpha)$. Let $s_0 = \emptyset$; having defined $s_\alpha$, let $s_{\alpha+1} = s_\alpha \,^\frown (\hat{n}_{\alpha+1}, \hat{n}'_{\alpha+1})$. Finally, having defined $s_\alpha$ for each $\alpha < \delta$ for $\delta < \kappa$ limit, since $|\delta| < \mathfrak{t}_{\hat{V}}$ we can choose $s_+$ an upper bound for $(s_\alpha: \alpha < \delta)$. Let $\hat{n}$ be greatest so that $s_+(\hat{n})(0) < \hat{n}_\delta$ and $s_+(\hat{n})(1) > \hat{n}'_\delta$; let $s_\delta = s_+ \restriction_{\hat{n}} \,^\frown (\hat{n}_\delta, \hat{n}'_\delta)$.
	
	Since $\kappa < \mathfrak{t}_{\hat{V}}$ we can choose $s$ an upper bound on $(s_\alpha: \alpha < \kappa)$. Choose $\lambda$ regular and $(\hat{b}_\alpha: \alpha < \lambda)$ a strictly decreasing sequence from $\hat{\omega}$, which is cofinal above $(\hat{a}_\alpha: \alpha < \kappa)$, and such that $\hat{b}_0= \hat{\lg}(s)-1$.
	
	Then the sequences $(\hat{m}_\alpha: \alpha < \lambda_0)$ and $(s(\hat{b}_\alpha, 0): \alpha < \lambda)$ are cofinal in each other, so $\lambda_0 = \lambda$; and the sequences $(\hat{m}'_\alpha: \alpha < \lambda_0)$ and $(s(\hat{b}_\alpha, 1): \alpha < \lambda)$ are cofinal in each other, so $\lambda_1 = \lambda$.
\end{proof}

Note that in the following definition, we will eventually be proving that $\min(\mathfrak{p}_{\hat{V}}^+, \mathfrak{t}_{\hat{V}}) = \mathfrak{p}_{\hat{V}} = \mathfrak{t}_{\hat{V}}$.

\begin{definition}
	For $\kappa < \min(\mathfrak{p}_{\hat{V}}^+, \mathfrak{t}_{\hat{V}})$ regular, define $\mbox{lcf}_{\hat{V}}(\kappa)$ to be the unique regular $\lambda$ with $(\kappa, \lambda) \in \mathcal{C}_{\hat{V}}$ (which is also the unique regular $\lambda$ with $(\lambda, \kappa) \in \mathcal{C}_{\hat{V}}$). 
\end{definition}

Note that by definition of $\mathfrak{p}_{\hat{V}}$ there is some $\kappa \leq \mathfrak{p}_{\hat{V}}$ such that either $(\kappa, \mathfrak{p}_{\hat{V}}) \in \mathcal{C}_{\hat{V}}$ or else $(\mathfrak{p}_{\hat{V}}, \kappa) \in \mathcal{C}_{\hat{V}}$. If $\mathfrak{p}_{\hat{V}} < \mathfrak{t}_{\hat{V}}$, then $\kappa = \mbox{lcf}_{\hat{V}}({\mathfrak{p}_{\hat{V}}})$ and thus both occur.

Thus, for the contradiction, it suffices to show if $\mathfrak{p}_{\hat{V}} < \mathfrak{t}_{\hat{V}}$, then for all $\kappa \leq \mathfrak{p}_{\hat{V}}$, we have that $(\kappa, \mathfrak{p}_{\hat{V}}) \not \in \mathcal{C}_{\hat{V}}$.

The following easy case corresponds to Lemma 6.1 of \cite{pEqualsTref}.

\begin{lemma}\label{noSmallSymCuts}
	Suppose $\mathfrak{p}_{\hat{V}} < \mathfrak{t}_{\hat{V}}$. Write $\kappa = \mathfrak{p}_{\hat{V}}$. Then $(\kappa, \kappa) \not \in \mathcal{C}_{\hat{V}}$.
\end{lemma}
\begin{proof}
	Suppose it were, say via $(\hat{a}_\alpha: \alpha < \kappa)$, $(\hat{b}_\alpha: \alpha < \kappa)$. Write $\hat{N}_* = \hat{b}_0+1$. Let $(\hat{T}, \hat{<})$ be the tree of all sequences $s$ in $(\hat{N}_* \times \hat{N}_*)^{<\hat{N}_*}$ such that for all $\hat{n} < \hat{m} < \hat{\lg}(s)$, $s(\hat{n})(0) < s(\hat{n})(1) < s(\hat{m})(1) < s(\hat{m})(0)$. Using the techniques of the previous proofs it is easy to define $(s_\alpha: \alpha < \kappa)$ an increasing sequence from $\hat{T}$ such that if we set $\hat{n}_\alpha = \hat{\lg}(s_\alpha)$, then $s_\alpha(\hat{n}_\alpha-1) = (\hat{a}_\alpha, \hat{b}_\alpha)$. Then since $\kappa < \mathfrak{t}_{\hat{V}}$ is regular we can choose an upper bound $s$ for $(s_\alpha: \alpha < \kappa)$. Then $s(\hat{\lg}(s)-1)(0)$ is in the gap $(\hat{a}_\alpha: \alpha < \kappa), (\hat{b}_\alpha: \alpha < \kappa)$; but this was supposed to be a cut.
\end{proof}

Before finishing, we will want the following standard fact. It is listed as Fact 8.4 of \cite{pEqualsTref}.

\begin{lemma}\label{Combinatorics}
	For every $\kappa$, there is some map $g: [\kappa^+]^2 \to \kappa$ such that for every $X \subseteq \kappa^+$, if $|X| = \kappa^+$ then $|g[X^2]| = \kappa$.
\end{lemma}
\begin{proof}
Choose $g$ so that for all $\gamma < \beta < \alpha$, $g(\gamma, \alpha) \not= g(\beta, \alpha)$ (this is possible since for all $\alpha < \kappa^+$, there is an injection from $\alpha$ to $\kappa$). Suppose $X \subseteq \kappa^+$ has size $\kappa^+$. Then we can choose $\alpha \in X$ such that $|\alpha \cap X| = \kappa$. Then for all $\beta, \gamma \in \alpha \cap X$ distinct, $g(\beta, \alpha) \not= g(\gamma, \alpha)$; hence $|g[X]^2] = \kappa$.
\end{proof}

To finish the proof of Theorem~\ref{pVEqualsTv}, it suffices to establish the following lemma; it corresponds to Theorem 8.1 of \cite{pEqualsTref}.

\begin{lemma}\label{KeyLemma}  Suppose $\mathfrak{p}_{\hat{V}} < \mathfrak{t}_{\hat{V}}$; write $\lambda = \mathfrak{p}_{\hat{V}}$, and let $\kappa < \lambda$ be regular. Then $(\kappa, \lambda) \not \in \mathcal{C}_{\hat{V}}$.
\end{lemma}
\begin{proof} 
	We proceed like in the proof of Lemma~\ref{noSmallSymCuts}, but with a more inspired tree $\hat{T}$. Towards a contradiction, let $(\hat{a}_\alpha: \alpha < \kappa)$, $(\hat{b}_\beta: \beta < \lambda)$ be a $(\kappa, \lambda)$-cut. Let $\hat{N}_* = \hat{b}_0+1$.  Also, choose  a function $g: [\kappa^+]^2 \to \kappa$ as in Lemma~\ref{Combinatorics}. Extend $g$ to a function from $[\lambda]^2$ to $\kappa$ arbitrarily.
	
	Now, define $\hat{T}$ to be the tree of all sequences $s = (\hat{e}_{\hat{n}}, \hat{D}_{\hat{n}}, \hat{g}_{\hat{n}}: \hat{n} < \hat{n}_*) \in \hat{V}$ of length $\hat{n}_* < \hat{N}_*$, satisfying:
	
	\begin{enumerate}
		
		\item $(\hat{e}_{\hat{n}}: \hat{n} < \hat{n}_*)$ is a decreasing sequence with $\hat{e}_0 < \hat{N}_*$;
		\item For all $\hat{n} < \hat{n}_*$, we have that $\hat{D}_{\hat{n}} \subseteq \hat{n}$, and $\hat{g}_{\hat{n}}: [\hat{D}_{\hat{n}}]^2 \to \hat{e}_{\hat{n}}$;
		
		\item If $\hat{n}+1 < \hat{n}_*$ then $\hat{g}_{\hat{n}}$ and $\hat{g}_{\hat{n}+1}$ agree on $[\hat{D}_n \cap \hat{D}_{n+1}]^2$.
	\end{enumerate}

So as $\hat{n}$ increases, $\hat{g}_{\hat{n}}$ is squeezing pairs from $\hat{D}_{\hat{n}}$ into the shrinking space $\hat{e}_{\hat{n}}$. 

Suppose $\beta_* < \lambda$. Then say that the increasing sequence $(s_\beta: \beta \leq \beta_*)$ from $\hat{T}$ is nice if, writing $\hat{d}_\alpha = \lg(s_\alpha)$ for each $\beta \leq \beta_*$ and writing $s_{\beta_*}(\hat{n}) = (\hat{e}_{\hat{n}}, \hat{D}_{\hat{n}}, \hat{g}_{\hat{n}})$ for each $\hat{n} < \hat{d}_{\beta_*}$, the following conditions are met:

\begin{enumerate}

\item[4.] For all $\beta < \beta_*$, $\hat{e}_{\hat{d}_{\beta}} = \hat{b}_\beta$;

\item[5.] For all $\beta< \beta_*$ and for all $\hat{d}_\beta < \hat{n} < \hat{n}_*$, $\hat{d}_\beta \in \hat{D}_n$;

\item[6.] For all $\beta < \beta' < \beta_*$ and for all $\hat{d}_{\beta'} < \hat{n} < \hat{n}_*$, $\hat{g}_{\hat{n}}(\hat{d}_\beta, \hat{d}_{\beta'}) = \hat{a}_{g(\beta, \beta')}$.
\end{enumerate}

Also, for limit ordinals $\delta \leq \lambda$, say that the increasing sequence $(s_\beta: \beta < \delta)$  from $\hat{T}$ is nice each proper initial segment is.

\noindent \textbf{Claim.} There is a nice sequence $(s_\beta: \beta < \lambda)$ from $\hat{T}$.
\vspace{1 mm}

Before proving the claim, we indicate why it suffices. Let $(s_\beta: \beta < \lambda)$ be a nice sequence from $\hat{T}$. Since $\lambda = \mathfrak{p}_{\hat{V}} < \mathfrak{t}_{\hat{V}}$, we can find an upper bound $s_\lambda$ to $(s_\beta: \beta < \lambda)$ in $\hat{T}$. Write $\hat{d}_\beta = \lg(s_\beta)$ for each $\beta \leq \lambda$, and write $s_\lambda(\hat{n}) = (\hat{e}_{\hat{n}}, \hat{D}_{\hat{n}}, \hat{g}_{\hat{n}})$ for each $\hat{n} < \hat{d}_{\lambda}$. The idea is to find some $\gamma < \gamma' < \kappa^+$ and some $\hat{d}_{\gamma'} < \hat{n} < \hat{n}_*$ such that $\hat{n}$ is small enough that $\hat{g}_{\hat{n}}(\hat{d}_\gamma, \hat{d}_{\gamma'}) = \hat{a}_{g(\gamma, \gamma')}$, and such that $\hat{n}$ is large enough that $\hat{e}_{\hat{n}} \leq \hat{a}_{g(\gamma, \gamma')}$. This will be a contradiction.

Formally, choose a decreasing sequence $(\hat{k}_\alpha: \alpha < \kappa)$ with $\hat{k}_0 = \hat{d}_\lambda$ so that $(\hat{d}_\beta: \beta < \lambda), (\hat{k}_\alpha: \alpha < \kappa)$ is a cut; this is possible by uniqueness of $\mbox{lcf}_{\hat{V}}(\lambda) = \kappa$. Note that for each $\gamma < \kappa^+$, we can find some $\alpha_\gamma < \kappa$ such that whenever $\hat{d}_\gamma \leq \hat{n} \leq \hat{k}_{\alpha_\gamma}$, we have that $\hat{d}_\gamma \in \hat{D}_{\hat{n}}$ (otherwise, the least $\hat{n} \geq \hat{d}_\gamma$ with $\hat{d}_\gamma \not \in \hat{D}_{\hat{n}}$ would fill the cut $(\hat{d}_\beta: \beta < \lambda), (\hat{k}_\alpha: \alpha < \kappa)$). Then we can find some $\alpha < \kappa$ such that $\{\gamma < \kappa^+: \alpha_\gamma = \alpha\}$ has size $\kappa^+$.  Let $\alpha' < \kappa$ be large enough so that  $\hat{e}_{\hat{k}_\alpha} \leq \hat{a}_{\alpha'}$ (if there were no such $\alpha'$ then $\hat{e}_{\hat{k}_\alpha}$ would fill the cut $(\hat{a}_\alpha: \alpha < \kappa), (\hat{b}_\beta: \beta < \lambda)$). Now by choice of $g$, there are $\gamma < \gamma' \in \Gamma$ with $g(\gamma, \gamma') \geq \alpha'$. Now $\hat{g}_{\hat{k}_\alpha}(\hat{d}_\gamma, \hat{d}_{\gamma'}) = \hat{a}_{g(\gamma, \gamma')}$ by condition 3; but $\hat{a}_{g(\gamma, \gamma')} \geq \hat{a}_{\alpha'} \geq \hat{e}_{\hat{k}_\alpha}$, contradicting that $\hat{g}_{\hat{k}_\alpha}: [\hat{D}_{\hat{k}_\alpha}]^2 \to \hat{e}_{\hat{k}_\alpha}$.

So it suffices to prove the claim. We define our nice sequence  $(s_\beta: \beta <\lambda)$ inductively. At the stage $\beta = 0$, we just set $s_0 = \emptyset$. At limit stages, there is nothing to do.

Suppose we have constructed $(s_\beta: \beta < \beta_*)$, where $\beta_*< \lambda$ is a limit ordinal (i.e. we are at the successor of a limit stage). By Theorem~\ref{localSaturation}, we can find some upper bound $s_{\beta_*}$ to $(s_\beta: \beta < \beta_*)$ in $\hat{T}$, such that if we write $\hat{d}_\beta = \lg(s_\beta)$ for each $\beta \leq  \beta_*$, and write $s_{\beta_*} = (\hat{e}_{\hat{n}}, \hat{D}_{\hat{n}}, \hat{g}_{\hat{n}}: \hat{n} < \hat{d}_{\beta_*})$, then for all $\beta < \beta_*$ and for all $\hat{d}_{\beta} < \hat{n} < \hat{d}_{\beta_*}$, $\hat{d}_\beta \in \hat{D}_n$. Then $(s_\beta: \beta < \beta_* +1)$ is nice.

Finally, suppose we have constructed $(s_\beta: \beta < \beta_*+1)$ for some $\beta_*<  \lambda$. Write $\hat{d}_\beta = \lg(s_\beta)$ for each $\beta \leq \beta_*$, and write $s_{\beta_*} = (\hat{e}_{\hat{n}}, \hat{D}_{\hat{n}}, \hat{g}_{\hat{n}}: \hat{n} < \hat{d}_{\beta_*})$. Write $\hat{n} = \hat{d}_\beta$. Let $\hat{d}_{\beta_*+1} = \hat{n}+ 1$ and let $\hat{e}_{\hat{n}} = \hat{b}_{\beta}$. By Theorem~\ref{localSaturation} (using $\hat{\mathcal{P}}(\hat{D}_{\hat{n}-1})$ is pseudofinite), we can find some $\hat{D} \subseteq \hat{D}_{\hat{n}-1}$ such that $\hat{d}_\gamma \in \hat{D}$ for all $\gamma< \beta_*$, and such that for all $\hat{u} \in [\hat{D}]^2$, $\hat{g}_{\hat{n}-1}(\hat{u}) < \hat{e}_{\hat{n}} = \hat{b}_\beta$. (This uses that each $\hat{g}_{\hat{n}-1}(\hat{d}_\gamma, \hat{d}_{\gamma'}) = \hat{a}_{g(\gamma, \gamma')} < \hat{b}_\beta$.) Let $\hat{D}_{\hat{n}} = \hat{D} \cup \{\hat{n}\}$. By Theorem~\ref{localSaturation} again (using $\hat{e}_{\hat{n}}^{\hat{D} \times \hat{D}}$ is pseudofinite), we can find $\hat{g}_{\hat{n}}: [\hat{D}_{\hat{n}}]^2 \to \hat{e}_{\hat{n}}$ extending $\hat{g}_{\hat{n}-1} \restriction_{\hat{D}}$, such that $\hat{g}_{\hat{n}}(\hat{d}_\gamma, \hat{d}_{\beta_*}) = \hat{a}_{g(\gamma, \beta_*)}$ for every $\gamma < \beta_*$. Let $s_{\beta_*+1} = s_{\beta_*} \,^\frown\, (\hat{e}_{\hat{n}}, \hat{D}_{\hat{n}}, \hat{g}_{\hat{n}})$. Then $(s_\beta: \beta < \beta_*+2)$ is nice.
\end{proof}

This concludes the proof that $\mathfrak{p}_{\hat{V}} = \mathfrak{t}_{\hat{V}}$.

\section{$\mathfrak{p} = \mathfrak{t}$}\label{CardInvariants}
We begin the final leg of the proof of $\mathfrak{p} = \mathfrak{t}$ with the relevant definitions:

\begin{definition}
	\begin{itemize}
		\item Given $X, Y \subset \omega$, say that $X \subseteq_* Y$ if $X \backslash Y$ is finite.
		\item Given $\mathcal{B} = \{B_\alpha: \alpha < \kappa\}$ say that $\mathcal{B}$ has the strong finite intersection property if the intersection of finitely many elements from $\mathcal{B}$ is infinite. Say that $\mathcal{B}$ has a pseudo-intersection if there is some infinite $X \subset \omega$ with $X\subseteq_* B_\alpha$ for each $\alpha < \kappa$.
		\item Let $\mathfrak{p}$ be the least cardinality of a familiy $\mathcal{B}$ of subsets of $\omega$ with the strong finite intersection property but without an infinite pseudo-intersection.
		\item Say that $(X_\alpha: \alpha < \kappa)$ is a tower if each $X_\alpha \subseteq \omega$ is infinite, and $\alpha < \beta < \kappa$ implies $X_\alpha \supseteq_* X_\beta$.
		\item Let $\mathfrak{t}$ be the least cardinality of a tower with no pseudo-intersection.
	\end{itemize}
\end{definition}

Obviously $\mathfrak{p} \leq \mathfrak{t}$. See \cite{CombCharSurvey} for a survey on the classical theory of cardinal invariants of the continuum.

We will want the following definition.

To begin making connections with the previous section we observe the following lemma.

\begin{lemma}\label{pVtVEquivalents}
	Suppose $\hat{V} \models ZFC^-$ is $\omega$-nonstandard. Then the following are equivalent:
	
	\begin{itemize}
		\item[(A)] $\lambda < \mathfrak{p}_{\hat{V}}$.
		\item[(B)] $\lambda < \mathfrak{t}_{\hat{V}}$.
		\item[(C)] Whenever $(\hat{a}_\alpha: \alpha < \lambda)$ is a family from $[\hat{\omega}]^{<\hat{\aleph}_0}$ such that for all $\alpha_0, \ldots, \alpha_{n-1} \in \lambda$, $|\hat{a}_{\alpha_0} \hat{\cap} \ldots \hat{\cap} \hat{a}_{\alpha_{n-1}}|$ is nonstandard, then there is $\hat{a} \in [\hat{\omega}]^{<\hat{\aleph_0}}$ with $|\hat{a}|$ nonstandard, such that $\hat{a}\subseteq \hat{a}_\alpha$ for each $\alpha < \lambda$.
		\item[(D)] Whenever $(\hat{a}_\alpha: \alpha < \lambda)$ is a descending sequence of nonempty sets from $[\hat{\omega}]^{<\hat{\aleph}_0}$, there is some $\hat{m}<\hat{\omega}$ such that $\hat{m} \in \hat{a}_\alpha$ for each $\alpha < \lambda$.
	\end{itemize}
\end{lemma}
\begin{proof}
	(A) and (B) are equivalent by Theorem~\ref{pVEqualsTv}, and they imply the other items by Theorem~\ref{localSaturation}. (For (A) implies (C), note that we are requiring $\hat{a} \in \hat{\mathcal{P}}(\hat{a}_0)$, a pseudofinite set.) Also, clearly (C) implies (D). So it suffices to show that (D) implies (B).
	
	Suppose $(s_\alpha: \alpha < \lambda)$ is an increasing sequence from $\hat{n}^{<\hat{n}}$. Let $\hat{a}_\alpha = \{s \in \hat{n}^{\hat{n}-1}: s_\alpha \subseteq s\}$. Then by (D) (and applying an injection from $\hat{n}^{\hat{n}-1}$ to $\hat{n}'$ for large enough $\hat{n}'$) we can choose $s \in \hat{n}^{\hat{n}}$ with $s \in \hat{a}_\alpha$ for each $\alpha < \lambda$. Then $s$ is an upper bound on $(s_\alpha: \alpha < \lambda)$.
\end{proof}

We need one more lemma. It is implicit in the proof of Claim 14.7 of \cite{pEqualsTref}.

\begin{definition}
	Suppose $f, g: \omega \to [\omega]^{<\aleph_0}$ and $A \subseteq \omega$ is infinite. Then say that $f \leq_A g$ if $\{n \in A: f(n) \not \subseteq g(n)\}$ is finite.
\end{definition}

\begin{lemma}\label{pEqualsTTechnical}
	Suppose $\lambda <  \mathfrak{t}$ is an infinite cardinal and $A \subseteq \omega$ is infinite and $(f_\alpha: \alpha < \lambda)$ is a sequence from $([\omega]^{<\aleph_0})^\omega$ with $f_\alpha \geq_A f_\beta$ for all $\alpha < \beta < \lambda$. Suppose further that for each $\alpha < \lambda$, $\{m \in A: f_\alpha(m) = \emptyset\}$ is finite. Then there is some infinite $B \subseteq A$ and some $f: \omega \to [\omega]^{<\aleph_0}$ such that $f \leq_B f_\alpha$ for each $\alpha < \lambda$, and further $f(m) \not= \emptyset$ for each $m \in B$.
\end{lemma}

\begin{proof}
	For notational simplicity, we assume $A = \omega$.
	For each $\alpha < \lambda$ define $X_\alpha := \{\langle m, n\rangle: n \in f_\alpha(m)\}$; so $X_\alpha$ is an infinite subset of $\omega \times \omega$. Suppose $\alpha < \beta$; then there is $m_*$ so that for all $m\geq m_*$, $f_\alpha(m) \subseteq f_\beta(m)$. Hence $X_\alpha \backslash X_\beta \subseteq \bigcup_{m < m_*} \{m\} \times f_\alpha(m)$ is finite, so $X_\alpha \supseteq_* X_\beta$. Hence $(X_\alpha: \alpha < \lambda)$ is a tower; by hypothesis on $\lambda$ we can choose an infinite $X \subseteq \omega \times \omega$ such that $X \subseteq_* X_\alpha$ for each $\alpha < \lambda$. Define $f: \omega \to [\omega]^{<\aleph_0}$ by $f(m) = \{n: \langle m, n \rangle \in X\}$. (Each $f(m)$ is finite because $X \subseteq_* X_0$.) Let $B = \{m < \omega: f(m) \not= \emptyset\}$. Clearly this works.
\end{proof}

\newpage

Finally, the following is Theorem 14.1 of \cite{pEqualsTref}.

\begin{theorem}\label{pEqualsT}
	$\mathfrak{p} =\mathfrak{t}$.
\end{theorem}
\begin{proof}
	We know that $\mathfrak{p} \leq \mathfrak{t}$; suppose towards a contradiction that $\mathfrak{p} < \mathfrak{t}$. We can suppose that $\mathfrak{t} = 2^{\aleph_0} = 2^{<\mathfrak{t}}$ since if we force by the Levy collapse of $2^{<\mathfrak{t}}$ to $\mathfrak{t}$, this adds no new sequences of reals of length less than $\mathfrak{t}$, and so does not affect the values of $\mathfrak{p}$ and $\mathfrak{t}$. So henceforth we assume this.
	
	Our aim is to build a special ultrafilter $\mathcal{U}$ on $\omega$, such that if we set $\hat{V} =V^\omega / \mathcal{U}$ for some or any transitive $V \models ZFC^-$, then $\mathfrak{p}_{\hat{V}} \leq \mathfrak{p}$ and $\mathfrak{t}_{\hat{V}} \geq \mathfrak{t}$. In view of $\mathfrak{p}_{\hat{V}} = \mathfrak{t}_{\hat{V}}$ this clearly suffices for the contradiction.
	
	Inductively choose a tower $(A_\gamma: \gamma < \mathfrak{t})$ so that:
	
	\begin{itemize}
		\item[(1)] (This is the definition of tower) Each $A_\gamma$ is infinite and $\gamma < \gamma' < \mathfrak{t}$ implies $A_{\gamma'} \subseteq_* A_\gamma$.
		\item[(2)] For each $A \subseteq \omega$, there is some $\gamma< \mathfrak{t}$ such that either $A \subseteq X_\gamma$ or else $A \cap X_\gamma = \emptyset$; 
		\item[(3)] Suppose $(f_\alpha: \alpha < \lambda)$ is a sequence from $([\omega]^{<\aleph_0})^\omega$ of length $\lambda < \mathfrak{t}$, and for some $\gamma < 2^{\aleph_0}$, we have that for all $\alpha < \beta < \lambda$, $f_\beta \leq_{A_{\gamma}} f_\alpha$. 
		
		Then for some $\gamma_* \geq \gamma$, there is some $f: \omega \to [\omega]^{<\aleph_0}$ such that $\{m \in A_{\gamma_*}: f(m) = \emptyset\}$ is finite, and such that $f \leq_{A_{\gamma_*}} f_\alpha$ for each $\alpha < \lambda$.
	\end{itemize}
	
	This is straightforward, using Lemma~\ref{pEqualsTTechnical} and that $(2^{\aleph_0})^{<\mathfrak{t}} = \mathfrak{t}$. Let $\mathcal{U} $ be the set of all $A \subset \omega$ such that $A_\gamma \subseteq_* A$ for some $\gamma < 2^{\aleph_0}$. Then $\mathcal{U}$ is a nonprincipal ultrafilter, by (1) and (2). Let $V$ be a transitive model of $ZFC^-$; for instance, we can take $V = \HC$, the set of hereditarily countable sets. Let $\hat{V} = V^\omega/\mathcal{U}$.

	\vspace{1 mm}
	
	\noindent \textbf{Claim 1.} $\mathbf{p}_{\hat{V}} \leq \mathfrak{p}$.
	
	\noindent \emph{Proof of Claim 1.} Suppose $\lambda < \mathfrak{p}_{\hat{V}}$; we show $\lambda < \mathfrak{p}$. Let $\{B_\alpha: \alpha < \lambda\}$ be a family of subsets of $\omega$ with the strong finite intersection property. Define $f_\alpha: \omega \to [\omega]^{<\aleph_0}$ by $f_\alpha(m) = B_\alpha \cap m$; let $\hat{a}_\alpha = [f_\alpha]_\mathcal{U}$.  So each $\hat{a}_\alpha \in [\hat{\omega}]^{<\hat{\aleph}_0}$, and $\{\hat{a}_\alpha: \alpha < \lambda\}$ satisfies the hypothesis of Lemma~\ref{pVtVEquivalents}(C). Since $\lambda < \mathfrak{p}_{\hat{V}}$ there is $\hat{a} \in [\hat{\omega}]^{<\hat{\aleph}_0}$ with $|\hat{a}|$ nonstandard, and with $\hat{a} \subseteq \hat{a}_\alpha$ for each $\alpha < \omega$.
	
	Write $\hat{a} = [f]_{\mathcal{U}}$. For each $\alpha < \lambda$ there is some $\gamma < \mathfrak{t}$ such that $f \leq_{A_\gamma} f_\alpha$. Since $\mathfrak{t}$ is regular, we can choose $\gamma_*$ large enough so that $f \leq_{A_{\gamma_*}} f_\alpha$ for each $\alpha < \lambda$. Define $B \subseteq \omega$ by $B = \bigcup_{m \in A_{\gamma_*}} f(m)$. $B$ is infinite since $\{m  \in A_{\gamma_*}: |f(m)| \geq n\} \in \mathcal{U}$ for each $n < \omega$. Also, suppose $\alpha < \lambda$; choose $m_*$ large enough so that $f(m) \subseteq f_\alpha(m)$ for every $m \in A_{\gamma_*} \backslash m_*$. Then $B \backslash B_{\alpha} \subseteq \bigcup_{m \in A_{\gamma_*} \cap m_*} f(m)$ is finite, so $B \subseteq_* B_\alpha$. This shows that $\lambda < \mathfrak{p}$, concluding the proof of the claim.
	
	\vspace{1 mm}
	
	\noindent \textbf{Claim 2.} $\mathfrak{t}_{\hat{V}} \geq \mathfrak{t}$.
	
	\noindent \emph{Proof of Claim 2.}  Let $\lambda < \mathfrak{t}$ be given; we show $\lambda < \mathfrak{t}_{\hat{V}}$. It suffices to show (D) from Lemma~\ref{pVtVEquivalents} holds. So let $(\hat{a}_\alpha: \alpha < \lambda)$ be a descending sequence of nonempty sets from $[\hat{\omega}]^{<\hat{\aleph}_0}$; write $\hat{a}_\alpha = [f_\alpha]_{\mathcal{U}}$.
	
	Note that for each $\alpha < \beta < \lambda$ there is some $\gamma$ with $f_\alpha \geq_{A_\gamma} f_\beta$, and for each $\alpha < \lambda$ there is some $\gamma$ with $\{m \in A_\gamma: f_\alpha(m) = \emptyset\}$ finite. Since $\mathfrak{t}$ is regular we can choose $\gamma_*$ large enough so that $f_\alpha \geq_{A_{\gamma_*}} f_\beta$ for all $\alpha < \beta < 2^{\aleph_0}$, and such that $\{m \in A_{\gamma_*}: f_\alpha(m) = \emptyset\}$ is finite for each $\alpha < \lambda$. 
	
	By item (3) of the construction we can find $\gamma \geq \gamma_*$ and $f: \omega \to [\omega]^{<\aleph_0}$ such that $f(m) \not= \emptyset$ for all but finitely many $m \in A_\gamma$, and $f \leq_{A_{\gamma}} f_\alpha$ for each $\alpha < \lambda$. Let $\hat{a} = [f]_{\mathcal{U}}$; then $\hat{a}$ is nonempty, so any $\hat{m} \in \hat{a}$ is as desired.
\end{proof}

\bibliography{mybib}

\end{document}